\documentclass{amsart}

\usepackage[english]{babel}
\usepackage{amsthm,amssymb,latexsym}
\usepackage{mathtools}
\usepackage{tikz-cd}
\usetikzlibrary{decorations.markings,arrows,patterns}
\usepackage{pgfplots}
\usepackage[T1]{fontenc}
\usepackage[utf8]{inputenc}
\usepackage{lmodern}
\usepackage{graphicx}
\usepackage{color}
\usepackage{enumitem}
\usepackage[colorlinks=false]{hyperref}
\usepackage[all]{hypcap}

\usepackage{caption}
\usepackage[justification=centering,labelformat=simple]{subcaption}
\theoremstyle{plain}
\newtheorem{theorem}{Theorem}[section]
\newtheorem{lemma}[theorem]{Lemma}

\theoremstyle{definition}
\newtheorem{remark}[theorem]{Remark}
\newtheorem{definition}[theorem]{Definition}
\newtheorem{corollary}[theorem]{Corollary}

\newcommand{\Con}[1] { \mathrm{CH}_*(#1) }
\newcommand{\timestep}{h}
\newcommand{\stepfunction}{\tau}
\renewcommand{\phi}{\varphi}
\renewcommand{\epsilon}{\varepsilon}
\newcommand{\setof}[1]{{\{\,#1\,\}}}

\usepackage{mathptmx}
\DeclareMathAlphabet{\mathcal}{OMS}{cmsy}{m}{n}

\newcommand{\mvmap}{\rightrightarrows}

\newcommand\cA{\mathcal A}

\newcommand\cF{\mathcal F}

\newcommand\cN{\mathcal N}
\newcommand\cP{\mathcal P}
\newcommand\cX{\mathcal X}
\newcommand\cY{\mathcal Y}
\newcommand\cZ{\mathcal Z}

\newcommand\NN{\mathbb{N}}
\newcommand\RR{\mathbb{R}}

\newcommand\ZZ{\mathbb{Z}}

\newcommand{\cl}{\operatorname{cl}}
\newcommand{\gker}{\operatorname{gker}}
\newcommand{\im}{\operatorname{im}}

\newcommand{\id}{\operatorname{id}}
\newcommand{\dom}{\operatorname{dom}}
\newcommand{\Inv}{\operatorname{Inv}}
\newcommand{\Int}{\operatorname{int}}

\newcommand{\Ho}[1] { \operatorname{H}_*(#1) }
\newcommand{\Leray} { \operatorname{L} }

\newcommand{\enumeration}[1]{
  \begin{enumerate}[label=(\roman*), itemsep=0.35ex, topsep=0.7ex]
    #1
  \end{enumerate}
}

\begin{document}

\author{Konstantin Mischaikow}
\address{Department of Mathematics,
Rutgers, The State University of New Jersey,
  110 Frelinghuysen Rd,
Piscataway, NJ  08854-8019, USA}
\email{mischaik@math.rutgers.edu}
\author{Marian Mrozek}
\address{
  Division of Computational Mathematics, Faculty of Mathematics and Computer Science,
  Jagiellonian University,
  ul. \L{}ojasiewicza 6, 30-348~Krak\'ow, Poland
}
\email{mrozek@ii.uj.edu.pl}
\author{Frank Weilandt}
\address{
  Division of Computational Mathematics, Faculty of Mathematics and Computer Science,
  Jagiellonian University,
  ul. \L{}ojasiewicza 6, 30-348~Krak\'ow, Poland
}
\email{weilandt@ii.uj.edu.pl}

\title[Discretization stragies for Morse decompositions]
{Discretization strategies for computing Conley indices and Morse decompositions of flows}

\begin{abstract}
Conley indices and Morse decompositions of flows can be found by using algorithms which
rigorously analyze discrete dynamical systems. This usually involves integrating a time
discretization of the flow using interval arithmetic.
We compare the old idea of fixing a time step as a parameters to
a time step continuously varying in phase space.
We present an example where this second strategy necessarily yields better
numerical outputs and prove that our outputs yield a valid Morse
decomposition of the given flow.
\end{abstract}

\thanks{K.~Mischaikow was
partially supported by NSF grants NSF-DMS-0915019, 1125174, 1248071, and contracts from AFOSR and DARPA.
M.~Mrozek was partially supported by the EU {\sc Toposys} project FP7-ICT-318493-STREP and
by ESF under the ACAT Research Network Programme. F.~Weilandt was supported by the EU {\sc Toposys} project FP7-ICT-318493-STREP
and by the FNP MPD program \emph{Geometry and Topology in Physical Models}.}

\maketitle

\section{Introduction}

While the numerical approximation of ordinary differential equations has a long history, the systematic study of how to compute time invariant structures began in the 1980s.
Rigorous computations of these structures is an even more recent phenomenon
(see \cite{AKP,DJ} and references therein).
These latter efforts can be roughly divided into two approaches: direct computation of invariant sets, e.g., periodic orbits, heteroclinic and homoclinic orbits, invariant manifolds, and a more indirect approach based on identification of isolating neighborhoods.
The advantages of the direct approach are clear: it leads to efficient algorithms and provides precise numerical bounds on solutions.
The disadvantage is that, in general, it does not provide information about the global structure of the dynamics and the results are sensitive to changes in parameters.
As indicated below, the indirect approach overcomes these disadvantages, however, the construction of appropriate methods of approximation has proven to be a significant challenge.

For the sake of simplicity consider a smooth  ordinary differential equation
 \begin{equation}
 \label{eq:ode}
 \dot{x} = v(x),\quad x\in \RR^d
 \end{equation}
that generates a flow $\varphi\colon \RR\times \RR^d\to \RR^d$.
The focus of this paper is on obtaining a finite representation of $\varphi$ from which information concerning the structure of the dynamics can be obtained.
A theoretical resolution to this challenge is as follows.

Recall that given $N\subset \RR^d$ the \emph{maximal invariant set} in $N$ under the flow $\phi$ is
\[
\Inv(N,\varphi) := \{ x \in N  \mid  \varphi(\RR,x) \subset N \}.
\]
A compact set $N\subset \RR^d$ is an \emph{isolating neighborhood} under $\varphi$ if $\Inv(N,\varphi) \subset \Int{N}$. Then $\Inv(N,\varphi)$ is called an \emph{isolated invariant set}.
The Conley index (see Section~\ref{sec:ConleyIndex}) is an algebraic topological invariant that provides information about the existence and structure of isolated invariant sets.

To perform computations, we make use of a discretization in time.
Choose $h>0$ and define the diffeomorphism $\varphi_h\colon \RR^d \to \RR^d$ by
\[
\varphi_h(x) := \varphi(h,x).
\]
In analogy with the setting of flows, the maximal invariant set in $N$ under $\varphi_h$ is defined by
\[
\Inv(N,\varphi_h) := \{x\in N\mid \varphi_h^n(x)\in N,\ \forall n\in \ZZ\}
\]
from which one can define isolating neighborhoods and isolated invariant sets.

A fundamental result \cite[Theorem 1]{MM1} guarantees that $S\subset \RR^d$ is an isolated invariant set for $\varphi$ if and only if $S$ is an isolated invariant set for $\varphi_h$, independent of the choice of $h>0$.
Furthermore, according to \cite[Theorem 2]{MM1} the Conley index of $S$ computed using $\varphi_h$ determines the Conley index for $S$ under $\varphi$.
Thus from a mathematical perspective no information is lost by studying the map $\varphi_h$ as opposed to the flow $\varphi$.

To obtain a representation of $\varphi_h$ with which we can compute, we discretize
the phase space.
For the sake of simplicity of presentation, we assume that we are interested in understanding the global dynamics of $\varphi$ or equivalently $\varphi_h$ restricted to a rectangular isolating neighborhood $X\subset \RR^d$.
This allows us to discretize $X$ using a cubical grid $\cX$  (see Section~\ref{sec:numerics}) with diameter less than some given $\epsilon >0$.
The dynamics of $\varphi_h$ can be encoded using a multivalued map $\cF_\rho\colon\cX \mvmap\cX$ defined by
\[
\cF_\rho(\xi) := \setof{ \xi'\in\cX \mid B_\rho(\varphi_h(\xi))\cap\xi' \neq \varnothing }
\]
where $B_\rho(\varphi_h(\xi)) := \setof{x\in X\mid \inf_{y\in \varphi_h(\xi)}\setof{\| x-y\|} <\rho}$.

Making use of the fact that $\cF_\rho$ can be interpreted as a directed graph, a wide variety of efficient algorithms have been developed for finding isolating neighborhoods and Morse decompositions for $\varphi_h$ (see~\cite{AKKMOP},\cite{KMV1}) and for computing the associated Conley indices \cite{KMM,MM2,Sz}.
Furthermore, by choosing finer grids, i.e., letting $\epsilon \to 0$ and making better approximations, i.e., letting $\rho\to 0$, one can recover any isolated invariant set \cite{KMV1}.
More generally, one can find all attractor-repeller pairs or equivalently all Morse decompositions and in the limit recover Conley's fundamental decomposition theorem \cite{BK,KMV1}.
Thus, from a theoretical perspective this approach allows us to recover the local and global dynamics associated with isolated invariant sets generated by ordinary differential equations.

However, in concrete applications there are significant technical obstructions.
Observe that there are three explicit parameters in the above mentioned approach: the time step $h$, the grid size $\epsilon$, and the accuracy of the approximation of the dynamics $\rho$.
Furthermore, the dynamics that can be extracted  is quite sensitive to the particular choices of these parameters.
For example, for a fixed $\epsilon$, if $h$ is too small then $\xi\in \cF_\rho(\xi)$ for every $\xi\in\cX$ and every $\rho\geq 0$.
In this case no isolating neighborhood can be extracted and hence there is no interesting dynamics that can be resolved.
Notice that the number of grid elements increases rapidly as an inverse function of $\epsilon$ and thus it becomes computationally prohibitive to choose $\epsilon$ too small.
Observe that if $v(x)\neq 0$ for all $x\in\xi$, then it is reasonable to assume that for a sufficiently large $h$, $\xi\not\in \cF_\rho(\xi)$, at which point one may hope to be able to start extracting interesting dynamical features.
Clearly, the size of the parameter $h$ depends on $\| v\|$,
the magnitude of the vector field.
The naive conclusion is that $h$ should be chosen to be large.
However, this leads to other difficulties.
We need to evaluate $\varphi_h$ which in practice requires us to integrate the differential equation \eqref{eq:ode} over time $h$.
Simple Gronwall inequalities imply that, in general, the numerical bounds on errors associated with this integration will grow exponentially in $h$. Suggesting that large $h$ leads to exponentially large $\rho$, which in turn leads to large images of $\cF_\rho$.  Larger images of $\cF_\rho$ imply that less information related to the dynamics of $\varphi$ can be extracted.

To address these issues, in this paper we consider a more general approach to encoding the dynamics. Let $\tau\colon X\to (0,\infty)$ be a continuous function and consider the map $\varphi_\tau\colon X\to \RR^d$ defined by
\[
\varphi_\tau(x) := \varphi(\tau(x),x).
\]
In principle this allows us to vary the integration time over the phase space, which allows us to compensate for the varying magnitude of the vector field and  sensitivity of propagation of numerical errors.
A brief outline of the paper is as follows.

Section~\ref{sec:numerics} presents numerical examples that contrast results using  fixed time steps $\varphi_h$ and spatially varying time steps $\varphi_\tau$.
The goal is not to provide new results concerning the dynamics of a particular ordinary differential equation, but rather to demonstrate the potential usefulness of this approach.

We use the results of Section~\ref{sec:numerics} as motivation for the central mathematical results of this paper.
As indicated above, \cite[Theorem 1]{MM1} guarantees that on the level of isolated invariant sets the dynamics of $\varphi_h$ captures the dynamics of $\varphi$.
Since this is not always true for a $\varphi_\tau$ with varying time steps we prove Lemma~\ref{criterionEqualInv}, which provides a criterion under which we obtain the desired isolation.

In Section~\ref{sec:ConleyIndex}, we show that computing the Conley index of
an isolated invariant set $S$  for
$\varphi_\stepfunction$ and $\phi$ provides the
Conley index for $S$ under the flow $\varphi$.
This generalizes~\cite[Theorem 2]{MM1}.
In Section~\ref{sec:MD}, we prove that a Morse decomposition for $\phi_\tau$ is a Morse decomposition
for the flow $\phi$ if all the Morse sets are invariant under $\phi$,
thereby showing that our algorithm does compute what we are interested in.

Varying the time step in space was first proposed in
\cite{CMLZ}, but with less attention for the mathematical background,
which we present here.

This article uses the following notations:
$\RR_{\geq 0} := \{ x \in \RR \mid x \geq 0 \}$,
$\RR_+ := \{ x \in \RR \mid x > 0 \}.$
We recall definitions where they are needed for our statements.
Their proofs require some background in dynamical systems,
which can be found in \cite{T} and other textbooks about the field.

\section{A numerical example}
\label{sec:numerics}

\subsection{Combinatorial representations}

Given $d>0$ and a tuple $s \in \RR_+^d$ representing a box size,
the space $\RR^d$ is the union of the following cubes (more precisely,
products of intervals):
\[ \cY = \left\{ \prod_{i=1}^d [ m_i s_i, (m_i+1) s_i] \mid m \in \ZZ^d \right\}.\]
Let $\cX \subset \cY$ be a finite subset.
For a set $\cA \subset \cX$, its \emph{realization} is
$|\cA| \coloneqq \bigcup \setof{ \xi \mid \xi \in \cA }
\subset \RR^d$.
Our goal is to compute a meaningful
Morse decomposition within a compact set $X:=|\cX|$.

A map $\cF$ from $\cX$ to its power set is written as $\cF \colon \cX \mvmap \cX$.
We consider it as a multivalued map on $\cX$ or as a directed graph which
has an edge from $\xi \in \cX$ to $\xi'\in \cX$ iff $\xi' \in \cF(\xi)$.
The \emph{image} of $\cA \subset \cX$ is
$\cF(\cA) := \bigcup_{\xi \in \cA} \cF(\xi) \subset \cX.$
For a subset $Z \subset X$, we let $\Int_X(Z)$ be the interior of $Z$
with respect to the subspace topology of $X\subset\RR^d$.

\begin{definition}
\label{defCombEnclosure}
A map $\cF \colon \cX \mvmap \cX$
is a \emph{combinatorial enclosure} of $f\colon \RR^d \to \RR^d$
if for every $\xi \in \cX$:
$ f(\xi)\cap X \subset \Int_X{|\cF(\xi)|}$.
\end{definition}

\begin{definition}
Let $\cN \subset \cX$.
\enumeration{
\item A \emph{full solution} through $\xi$ in $\cN$ is
  a sequence $ \Gamma \colon \ZZ \to \cN$ such that
  $\Gamma(0) = \xi$ and \mbox{$\Gamma(n+1) \in \cF(\Gamma(n))$} for all $n\in\ZZ$.
\item
  $ \Inv(\cN,\cF) := \setof{ \xi \in \cN \mid \text{there is a full solution through $\xi$ in $\cN$} }. $
}
\end{definition}

We apply the algorithms and software described in~\cite{AKKMOP}
and~\cite{BGHKMOP} for finding
Morse decompositions (Definition~\ref{def:MorseDecomp}) as follows.
The software constructs a combinatorial
enclosure $\cF \colon \cX \rightrightarrows \cX$ of a discrete dynamical system
$f \colon \RR^d \to \RR^d$. Then it constructs
the strongly connected path components $\{ \cN(p) \mid p \in \cP \}$ of
the directed graph $\cF$ and builds the directed graph $\mathrm{MG}(\cF)$
with vertices $\cP$ and a directed edge from $p$ to $q$ if there
are $G \in \cN(p)$, $H \in \cN(q)$
and a path from $G$ to $H$ in $\cF$. This describes the dynamics of the discrete dynamical system $f$ because of the following Theorem 4.1 from~\cite{KMV1}.

\begin{theorem}
\label{theoremMorseDiscrete}
Let $X=|\cX|$ be an isolating neighborhood for $f$ and $S=\Inv(X,f).$ Then:
\enumeration{
\item Each set $|\cN(p)|$ is an isolating neighborhood for $f$.
\item $\mathrm{MG}(\cF)$ is a Morse graph for $f$ in the invariant set $S$
with Morse sets $\Inv(|\cN(p)|, f), p \in \cP$.
}
\end{theorem}

In the rest of this section, we give an example flow $\phi$ and show
that this algorithm yields a finer output when using $f(x)=\phi(\tau(x),x)$ than
when using $f(x)=\phi(h,x)$.
The justification that our outputs are indeed Morse decompositions for $\phi$
is formulated in Theorem~\ref{theoremMorseDecompCriteria}.
Every norm $\|.\|$ in this article is Euclidean,
i.e., $\|x\|^2 =\sum_{i=1}^dx_i^2$ for $x\in \RR^d$.

\subsection{Fixed time step}
\label{ssec:fixedStep}
The following ordinary differential equation is particularly challenging to analyze with a fixed
time step $\timestep$.
\begin{align}
\begin{split}
\label{ODE_2limitCycles}
\dot{x}_1 &= v_1(x) = -x_2 + x_1 (x_1^2 + x_2^2-\mu) (x_1^2+x_2^2-1) \\
\dot{x}_2 &= v_2(x) = x_1 + x_2 (x_1^2 + x_2^2-\mu) (x_1^2+x_2^2-1)
\end{split}
\end{align}
The equation has a fixed point $(0,0)$ and limit cycles with radius $1$ and $\sqrt{\mu}$ around the fixed point. This can be seen by its representation in polar coordinates:
\begin{align}
\begin{split}
\dot{r} &= r(r^2-\mu)(r^2-1) \\
\dot{\theta} &= 1
\end{split}
\end{align}

\begin{figure}
  \centering
  \begin{subfigure}[t]{0.45\textwidth}
  \centering
  \resizebox{\textwidth}{!}{
  \begin{tikzpicture}[scale=1]
    \begin{axis}[
      scale only axis,
      enlargelimits=false,
      axis on top
      ]
      \addplot graphics[xmin=-3,xmax=3,ymin=-3,ymax=3] {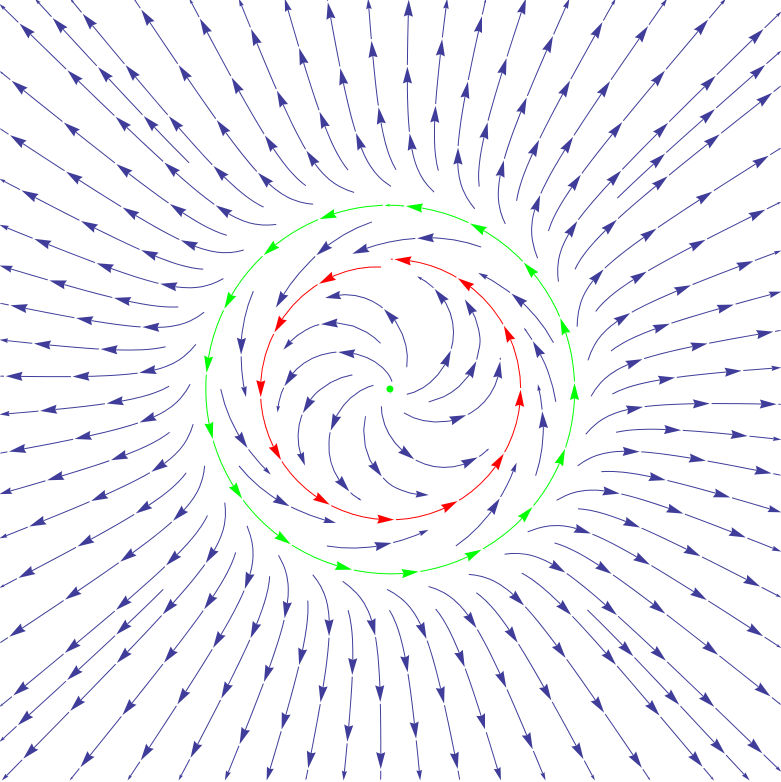};
    \end{axis}
  \end{tikzpicture}
  }
  \caption{Streamplot for Equation (\ref{ODE_2limitCycles})}
  \end{subfigure}
  \quad \quad
  \begin{subfigure}[t]{0.45\textwidth}
  \centering
  \resizebox{\textwidth}{!}{
  \begin{tikzpicture}[scale=1]
    \begin{axis}[
      scale only axis,
      enlargelimits=false,
      axis on top
      ]
      \addplot graphics[xmin=-3,xmax=3,ymin=-3,ymax=3] {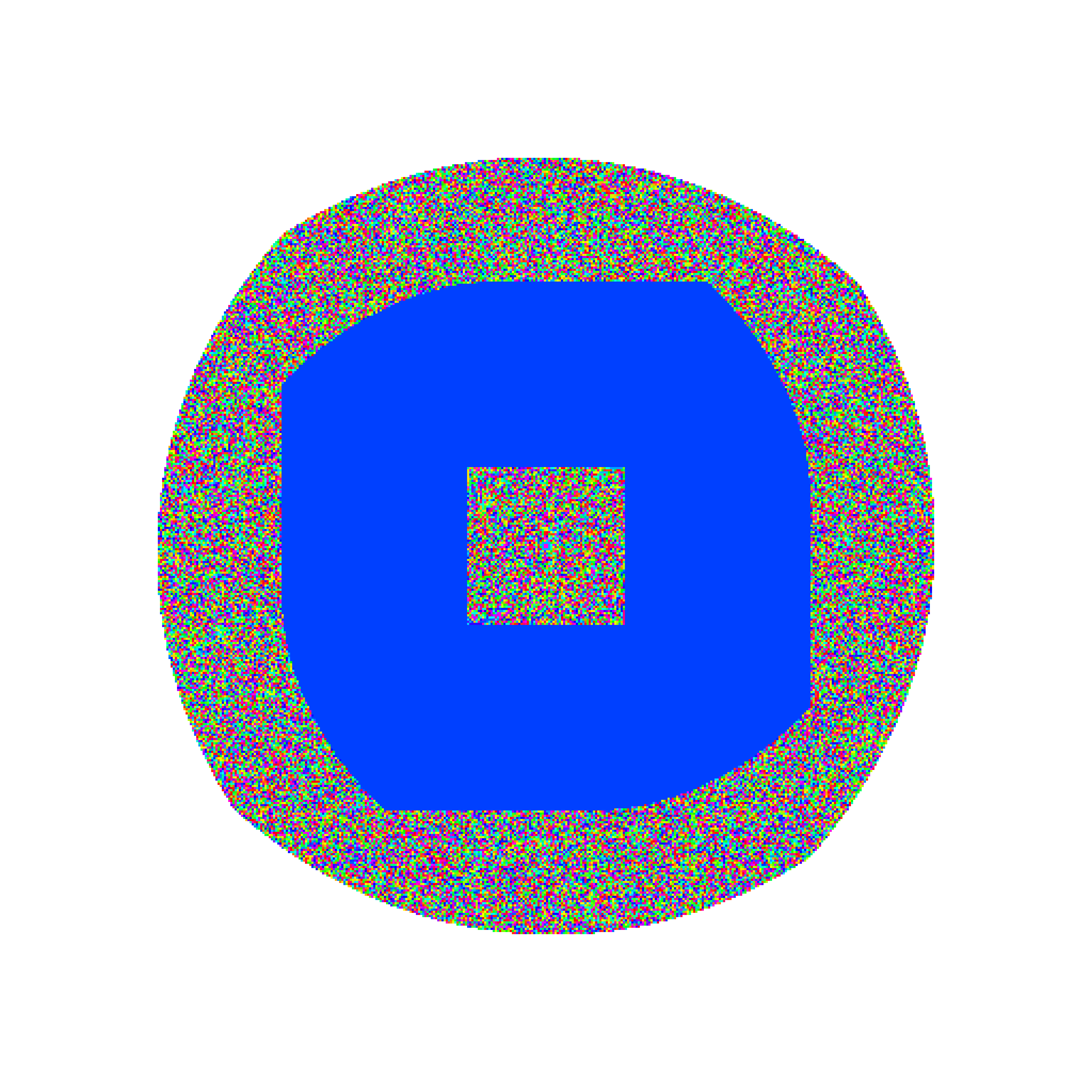};
    \end{axis}
  \end{tikzpicture}
  }
  \caption{The combinatorial Morse sets}
  \label{SubfigFixedT}
  \end{subfigure}
\caption{Visualization of the flow from Equation (\ref{ODE_2limitCycles}) for $\mu=2$
in $X=[-3,3]\times[-3,3]$ and
outputs when using the fixed time step $\timestep=0.0006$ as described
in Section~\ref{ssec:fixedStep}. There are $57\,673$ spurious
combinatorial Morse sets. Most of them consist of just one box. Calculating the Morse
graph was impossible because of memory problems.}
\label{FigVectorFieldVis}
\end{figure}

The norm of the vector field $v$ increases quickly away from the origin because
\[\| v(x) \| = \sqrt { r^2(r^2-\mu)^2(r^2-1)^2+r^2 } \approx r^5 \text{ for large $\|x\|$}.\]
Far away from the origin, the solutions
behave like the solutions of $\dot{r}(t) = r(t)^5$. This equation is solved by
\[ r(t) = \frac{1}{\sqrt{2}\sqrt[4]{\left(\sqrt{2} \cdot r_0\right)^{-4}-t}}, \]
where $r_0 = r(0).$
Note that the function $r$ is only defined for $t<(\sqrt{2}\cdot r_0)^{-4}=(x_1^2+x_2^2)^{-2}/4.$
The trajectories reach infinity after finite time. For a point $x$ with $\|x\|>\sqrt{\mu}$,
the solution for the system (\ref{ODE_2limitCycles})
is defined on a maximal open interval with right bound $T_+(x) < \infty.$
For $\timestep>T_+(x)$, the value $\phi_\timestep(x)$ is undefined.
Hence, also our integration algorithm
fails when the input is a box containing $x$ and we ask it to integrate until
time $\timestep > T_+(x).$
The fixed time step strategy can only be used with a parameter
$\timestep < \min \setof{T_+(x) \mid x \in X}$.
Therefore, when $X=[-3,3]^2$, we have to choose
\[ \timestep < T_+((3,3)) \approx (3^2+3^2)^{-2}/4 = 1/1296 \approx 0.0007716. \]
If $\timestep$ is chosen larger, each box $\xi \in \cX$ near the boundary of $X$ is assigned
an arrow in $\cF_\timestep$ to all of the other boxes because the algorithm
fails to find an enclosure of $\phi_\timestep(\xi)$ (which does not even exist).
This would lead to a very large
invariant part $\Inv(\cX, \cF_\timestep)$ and should therefore be avoided.
But when choosing $\timestep$ small enough, the directed graph $\cF_\timestep$ has a lot
of small strongly connected components whose corresponding invariant set of $\phi$
is empty (so-called \emph{spurious Morse sets}).
They have to occur when the time step
$\timestep$ is so small that $\phi(\timestep,\xi) \cap \xi \neq \varnothing$.
But this happens easily near the origin where $\|v(x)\|$ is small.

The outputs for $\mu=2$ are shown in Figure~\ref{SubfigFixedT}.
The algorithm as described in~\cite{BGHKMOP}
subdivided each dimension of $|\cX|$ into $2^9$
intervals of equal length.
Each colored region is a combinatorial Morse set $\cN(p), p \in \cP$.
The index set $\cP$ had $57\, 675$ elements. All but two of the
combinatorial Morse sets found are empty.

It took 27 seconds
on a laptop with an Intel i5 CPU to find the combinatorial Morse sets, but
the Morse graph could not be computed because $\cP$ was too large for the memory.
Also subdividing each dimension into $2^{12}$ intervals did give similarly bad outputs
(after 2396 seconds).

\subsection{Variable time step}
\label{ssec:variableStep}

We use the following heuristic for the time step function $\stepfunction$.
For each subdivision level, the Euclidean norm $\|s\|$ is the diagonal of each
of the congruent boxes $\xi \in \cX$.
Using parameters $D>1$ and $\delta>0$, we define the continuous function
\begin{equation}
\label{EqH}
\stepfunction \colon X \to \RR_+, \quad x \mapsto \frac{D\|s\|}{\|v(x)\|+\delta}.
\end{equation}
The idea is that the distance between $x$ and $\phi(\stepfunction(x),x)$ should be around $D$ box diagonals
-- using a first-order approximation $\phi(\stepfunction(x),x) \approx x + \stepfunction(x) v(x)$.
The number $\delta$ ensures that $\tau$ is also defined when $v(x)=0$.
Since $\stepfunction$
is usually not constant within a box $\xi$, the value $\stepfunction(\xi)$ is an interval. We can
find an enclosure of this interval by considering $\stepfunction$ as a function on
intervals and replacing $x$ by the box $\xi$ when calculating in interval arithmetic.
The software library CAPD \cite{CAPD} is used to construct
a combinatorial enclosure $\cF_\stepfunction \colon \cX \mvmap \cX$ of $\phi_\stepfunction$.
We use $D=4$ and $\delta=0.1$ in the following numerical example.

The varying time step
strategy proposed above yields the finest Morse decomposition of
the flow in $S = \Inv([-3,3]^2,\phi) = \setof{ x \mid \| x \| \leq \sqrt{\mu} }$
(more precisely, the finest one which does not
contain empty invariant sets). The output is shown in Figure~\ref{FigCompare2}.

Finding the Morse decomposition and the Morse graph took 29 seconds on the same hardware as in the
first example, but with each dimension subdivided into only $2^8$ intervals.
The algorithm additionally checks the correctness of the computed Morse decomposition.
This verification and its output in Figure~\ref{SubfigCriterionInv}
are described in Remark~\ref{numericsCriterionB}.

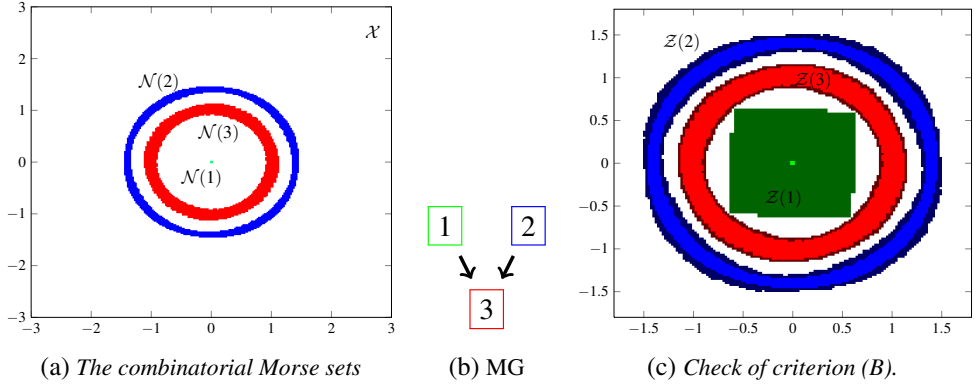
\begin{figure}
  \begin{subfigure}[t]{0.4\textwidth}
  \centering
  \resizebox{\textwidth}{!}{
  \begin{tikzpicture}[scale=1]
    \begin{axis}[
      scale only axis,
      enlargelimits=false,
      axis on top
      ]
      \addplot graphics[xmin=-3,xmax=3,ymin=-3,ymax=3] {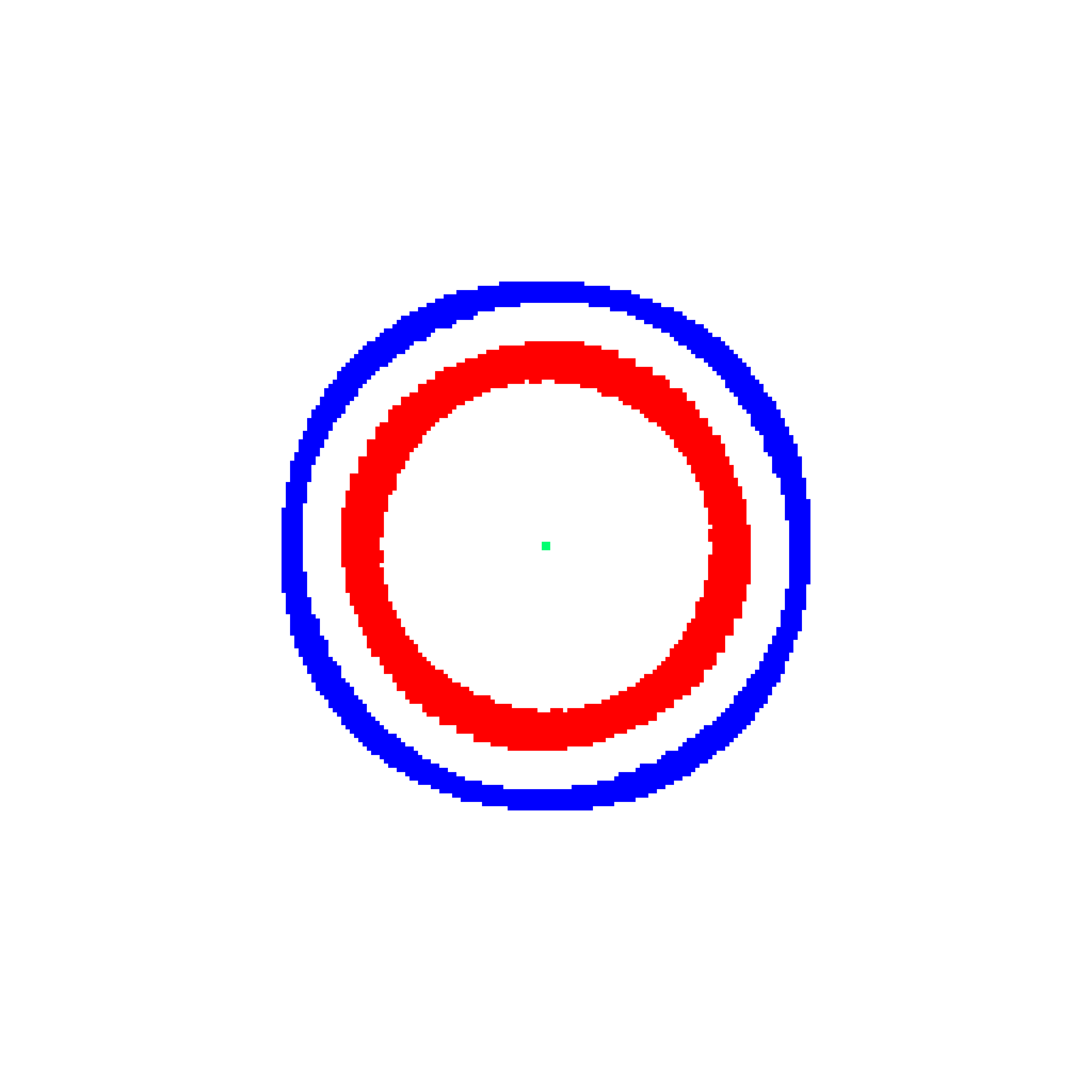};
    \end{axis}
    \node at (4,3.25) {\Large $\cN(1)$};
    \node at (4.4,4.3) {\Large $\cN(3)$};
    \node at (3,5.5) {\Large $\cN(2)$};
    \node at (8,6.7) {\Large $\cX$};
  \end{tikzpicture}
  }
  \caption{The combinatorial Morse sets}
  \label{SubfigVaryingT}
  \end{subfigure}
  \begin{subfigure}[t]{0.15\textwidth}
  \centering
  \resizebox{\textwidth}{!}{
  \begin{tikzpicture}[scale=0.5]
    \node (1) at (-1,2) {\fcolorbox{green}{white}{$1$}};
    \node (2) at (1,2) {\fcolorbox{blue}{white}{$2$}};
    \node (3) at (0,0) {\fcolorbox{red}{white}{$3$}} edge [very thick, <-] (1) edge [very thick, <-] (2);
  \end{tikzpicture}
  }
  \caption{$\mathrm{MG}$}
  \end{subfigure}
  \begin{subfigure}[t]{0.4\textwidth}
  \centering
  \resizebox{\textwidth}{!}{
  \begin{tikzpicture}[scale=1]
    \begin{axis}[
      scale only axis,
      enlargelimits=true,
      axis on top
      ]
      \addplot graphics[xmin=-1.5,xmax=1.5,ymin=-1.5,ymax=1.5] {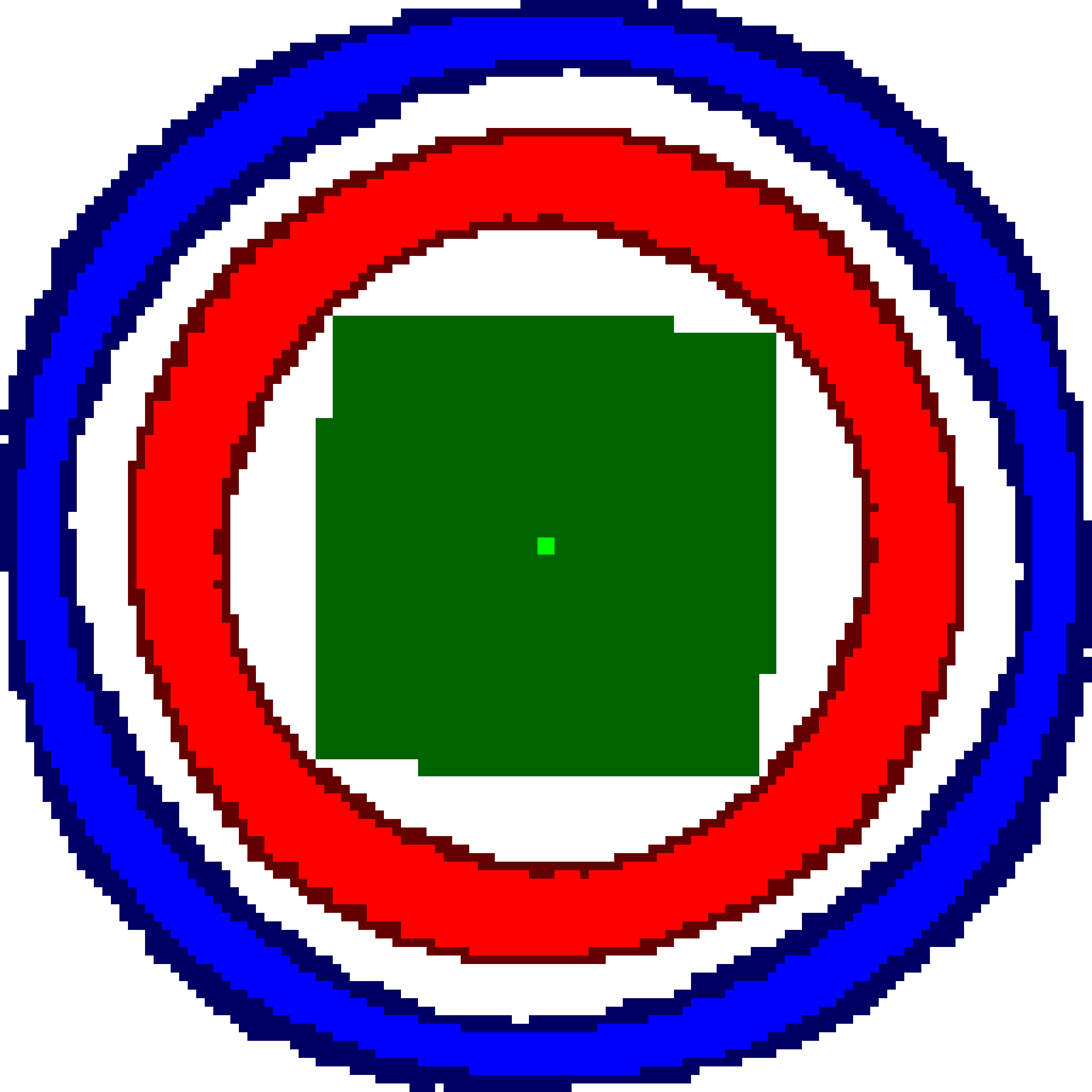};
    \end{axis}
    \node at (4,2.8) {\Large $\cZ(1)$};
    \node at (4.7,5.6) {\Large $\cZ(3)$};
    \node at (1.6,6.5) {\Large $\cZ(2)$};
  \end{tikzpicture}
  }
  \caption{Check of criterion~\ref{critVaryingStep}.}
  \label{SubfigCriterionInv}
  \end{subfigure}
\caption{Output for the same system as in Figure~\ref{FigVectorFieldVis},
  but using time step function $\stepfunction$
     from Equation~(\ref{EqH}) as proposed in Section~\ref{ssec:variableStep}.}
\label{FigCompare2}
\end{figure}

\section{Theoretical justification}
\label{sec:theory}

For the remainder of the paper we let $Y$ be a locally compact separable metric space.
Let $\phi\colon \RR\times Y \to Y$ be a flow.
We show when certain invariants for
$\phi_\stepfunction$ also yield the corresponding invariants for $\phi$.

\subsection{Isolating neighborhoods}

\begin{definition}
Let $f \colon Y \to Y$ be a continuous map.
\enumeration{
  \item A \emph{solution} through a point $x\in Y$ is a
  map $\gamma \colon \mathbb{Z} \to Y$ such that $\gamma(0)=x$ and
  $\gamma(n+1) = f(\gamma(n))$ for
  all $n\in\ZZ$.
  \item
  For $N\subset Y$ let
  $ \Inv(N,f) := \setof{ x \in N \mid \text{there is a solution $\gamma$ through $x$ such that $\gamma(\ZZ)\subset N$} }.$
  \item $N$ is an \emph{isolating neighborhood} of the
  \emph{isolated invariant set} $S$ if $S = \Inv(N,f) \subset \Int N.$
}
\end{definition}

Even though there is an inverse $\phi_{-\timestep}$ of $\phi_\timestep$ for any flow
$\phi$, there need not be an inverse for $\phi_\stepfunction$.

\begin{lemma}
\label{lemmaSubtrajectory}
Let $x\in Y$ and suppose that $\stepfunction(\phi(\RR,x)) \subset \RR^+$ is bounded.
Then there is a solution $\gamma \colon \mathbb{Z} \to Y$
of the discrete system $\phi_\stepfunction\colon Y \to Y$
such that $\gamma(0)=x$ and $\gamma(\ZZ) \subset \phi(\RR,x)$.
\end{lemma}

\begin{proof}
Define $\gamma(n) = \phi_\stepfunction^n(x)$ for $n\geq 0.$
Let $n<0$ and assume that $\gamma(n+1)$ is already constructed.
Then there is an $s\leq 0$ such that $\gamma(n+1)=\phi(s,x).$
Define the function
\[ g\colon \RR \to \RR, \quad t \mapsto \stepfunction(\phi(t,x))-s+t. \]
We have $g(s)>0$ and and since $t\mapsto \stepfunction(\phi(t,x))-s$ is bounded,
$g(t)<0$ for $t$ sufficiently small.
Hence, by the intermediate value theorem, there is a $t'<s$ such that $g(t')=0$.
Choose $\gamma(n):=\phi(t',x).$
Then
\[
  \phi(\tau(\gamma(n)),\gamma(n)) = \phi(s-t',\gamma(n)) = \phi(s,x) = \gamma(n+1).
\]
\end{proof}

The following theorem slightly generalizes \cite[Theorem~1]{MM1}.
\begin{theorem}
\label{theorem1n}
Let $S\subset Y$ be compact.
Then the following three conditions are equivalent:
\enumeration{
  \item $S$ is an isolated invariant set with respect to $\phi$.
  \item For every continuous map $\stepfunction\colon Y \to \RR^+$,
    $S$ is an isolated invariant set with respect to $\phi_\stepfunction$.
  \item There is a number $\timestep > 0$ such that $S$ is an isolated invariant
    set with respect to $\phi_\timestep$.
}
\end{theorem}
\begin{proof}
 Assume condition (i) holds and fix a continuous map $\stepfunction \colon Y\to\RR^+$. Choose  $N$,
 an isolating neighborhood for $S$ with respect to $\phi$. Obviously $\phi_\stepfunction(S)\subset S$.
 Using Lemma \ref{lemmaSubtrajectory}, for every $x \in S$ there is an $x'\in S$
 such that $\phi_\stepfunction(x') =x$. This means
 $S \subset \phi_\stepfunction(S)$.
 Hence $S$  is invariant with respect to $\phi _\stepfunction.$
 Let $T:=\sup\setof{\stepfunction(x)\mid x\in N} < \infty$. To see that $S$ is an isolated
 invariant set with respect to $\phi_\stepfunction$, consider the map
 \[
 \sigma\colon N\ni x\mapsto \sup\setof{t\in \RR_{\geq 0}\mid \phi([0,t],x)\subset  N}\in [0,\infty].
 \]
 One easily verifies that each $x\in S$ has a compact neighborhood
 $V_x$ such that $\sigma(V_x)\subset[T,\infty]$.
 Since $S$ is compact we can
 choose $M\subset N,$ a compact neighborhood of $S$ such that  $\sigma (x)\geq T$  for $x\in M.$
 We will show that $S=\Inv(M,\phi_\stepfunction)$.
 Obviously we have
 \[
 S=\Inv(N,\phi)=\Inv(M,\phi)\subset\Inv(M,\phi_\stepfunction).
 \]
 To show the opposite inclusion take $x\in\Inv(M,\phi_\stepfunction)$
 and let $\gamma \colon \ZZ \to Y$ be a solution of $\phi_\stepfunction$ through $x$.
 Let $x_n := \gamma(n)$ and $t_n:=\tau(x_n)$.
 Then
 \[
   x_{n+1}=\phi_\stepfunction^{n+1}(x)=\phi(t_n,\phi_\stepfunction^n(x))=
   \phi(t_n,x_n).
 \]
 By definition of $T$, we have $t_n\leq T$. Since $x_n\in M$, we have $\sigma(x_n)\geq T$.
 It follows that $\phi([0,t_n],x_n)\subset N$ for all $n$, and consequently
 $x\in\Inv(N,\phi)=S$. Thus implication $(i)\implies(ii)$ is proven.
 Implication $(ii)\implies(iii)$ is obvious because we can always take
 $\stepfunction$ to be a constant positive function.
 Implication $(iii) \implies (i)$ is part of Theorem~1 in \cite{MM1}.
\end{proof}

The following follows immediately from the implication $(iii) \implies (i)$.
\begin{corollary}[\cite{PP1}, Lemma 6]
\label{theorem_isolation}
Let $N \subset Y$ be an isolating neighborhood for $\phi_\timestep$. Then
\[ \Inv(N, \phi) = \Inv(N, \phi_\timestep).\]
\end{corollary}

\begin{remark}
There is no full analogue of Corollary~\ref{theorem_isolation} since
we cannot replace condition \emph{(iii)} in the theorem above by the statement
\emph{
\begin{itemize}
\item[(iii')] There is a continuous function $\stepfunction\colon X \to \RR^+$ such that $S$ is isolated invariant with respect to $\phi_\stepfunction$.
\end{itemize}
}

An example for which $(iii') \implies (i)$ is wrong
can be constructed by considering a system with a limit cycle
and using a time step function $\stepfunction$ which sends a certain
point $x'$ on the limit cycle
to itself under $\phi_\stepfunction$
by letting $\stepfunction(x')$ be the period of the orbit. With an appropriate choice of
$\stepfunction$ for $x$ near $x'$,
this yields an isolated fixed point $x'$ for $\phi_\stepfunction$,
but the set $S=\{x'\}$ would not be
invariant for $\phi$. We defer a detailed discussion of this example to
Remark \ref{remark_counterexample}, where it is discussed in the context of Morse decompositions.
\end{remark}

\subsection{Conley indices}
\label{sec:ConleyIndex}

Given an isolated invariant set of a flow or of a discrete dynamical system,
its Conley index is well-defined. Even though we are interested in the Conley index of a flow
(as described in \cite{C}), we propose using algorithms developed for
the calculation of Conley indices of maps.
Therefore, we need to consider both situations here.

\begin{definition}
\label{defIndexPairFlow}
Let $N\subset Y$ be compact and $S = \Inv(N,\phi)$.
A pair $(P_1,P_2)$ of compact sets $P_2 \subset P_1 \subset N$ is an \emph{index pair}
for $(S,\phi)$ in $N$ if it has the following properties:
\enumeration{
  \item If $x\in P_i,\, t>0,\, \phi([0,t],x) \subset N$, then $\phi([0,t],x) \subset P_i$;
  \item If $x \in P_1,\ t>0,\, \phi(t,x) \notin N$, then $\exists\, t' \in [0,t]: \phi(t',x) \in P_2,\, \phi([0,t'],x)\subset N$;
  \item $S \subset \Int(P_1 \setminus P_2)$.
}
\end{definition}

For an isolated invariant set $S$, an index pair exists and the following definition
does not depend on a particular choice of
the isolating neighborhood $N$ and the index pair $(P_1,P_2)$.
\begin{definition}
Let $(P_1,P_2)$ be an index pair for $(S,\phi)$ in $N$.
The \emph{homological Conley index} is $\Con{S,\phi} \coloneqq \Ho{P_1,P_2}.$
\end{definition}

There are several ways of defining the Conley index for maps.
Each definition uses a certain kind of index pair and
and equivalence relation on a map on this index pair (e.g., \cite{RS,FR,MM3}).
We use the following index pair from \cite[Definition 3.1]{MM2} here
and two popular equivalence relations in the subsections ~\ref{ssec:leray}
and~\ref{ssec:shift-equi}.
\begin{definition}
\label{defIndexPairMap}
Let $f\colon Y \to Y$ be a continuous map. Let $N\subset Y$ be compact
and $S=\Inv(N,f)$.
A pair $(P_1,P_2)$ of compact sets $P_2 \subset P_1 \subset N$ is an \emph{index pair}
for $(S,f)$ in $N$ if
\enumeration{
  \item $f(P_i) \cap N \subset P_i$,
  \item\label{propertyExitSet} $P_1 \setminus f^{-1}(P_1) \subset P_2$,
  \item $S \subset \Int{(P_1\setminus P_2)}.$
}
\end{definition}

Let $\mathrm{H}_*$ denote the homology functor with compact supports as defined in
\cite[Chapter 9]{massey1978}.
We use this version of the homology functor because of its good excision properties.
However, note that this homology theory coincides with singular homology
on compact neighborhood retracts,
in particular compact polyhedra and cubical sets.

Let $(P_1,P_2)$ be an index pair for $(S,f)$
in $N$. Define maps on topological pairs:
\begin{align*}
f_P \colon (P_1,P_2) \ni x &\mapsto f(x) \in (P_1 \cup (Y \setminus \Int{N}), P_2 \cup (Y \setminus \Int{N})); \\
i_P \colon (P_1,P_2) \ni x &\mapsto x \in (P_1 \cup (Y \setminus \Int{N}), P_2 \cup (Y \setminus \Int{N})).
\end{align*}
The induced homomorphism $\Ho{i_p}$ in relative homology is an isomorphism
because of excision (see, \cite[Theorem 9.3]{massey1978}).
The endomorphism $I_P = \Ho{i_P}^{-1} \circ \Ho{f_P}$ on $\Ho{P_1,P_2}$
is called the \emph{index map}.

In the next two subsections, we use the following notations for an
endomorphism $\alpha\colon V \to V$. Let $\dom(\alpha)=V$ denote the domain of $\alpha$
and $\gker(\alpha)$
the \emph{generalized kernel} defined by
\[\gker(\alpha)=\{ v\in V \mid \exists n\in\NN: \alpha^n(v)=0\}. \]

\subsubsection{Leray functor}
\label{ssec:leray}

Given an endomorphism $\alpha \colon V \to V$ of graded modules
(e.g. homology groups), the Leray functor $\Leray$ assigns to $\alpha$
an automorphism $\Leray(\alpha)$ as follows.
Let
\[\overline\alpha\colon V/\gker(\alpha) \to V/\gker(\alpha),\ [v] \mapsto [\alpha(v)].\]
Then let
\[ V':= \cap_{n\in\NN} {\overline\alpha}^n(V')\text{ and }
\Leray(\alpha) \colon V' \to V',\ [v]\mapsto \overline\alpha([v])=[\alpha(v)].\]
Observe that if $\alpha$ is already an automorphism, then $\Leray(\alpha)=\alpha$.
For a detailed definition of this functor,
we refer the reader to \cite[Section 4]{MM3}.
We define
$\alpha \sim \beta$ iff $\Leray(\alpha)$ and $\Leray(\beta)$ are conjugate.
Here we let the Conley index be the equivalence class of $\Leray(I_P)$ under this
equivalence relation.
This definition only depends on $S$ and not on the choice of $P$ or $N$.
\cite[Theorem 2.6]{MM3}.
For all the index maps $I_p$ considered in this article, the automorphism $\Leray(I_P)$
is conjugate to the identity as is shown later in this section.

The main ingredient for comparing the Conley index of $\phi$ with the
Conley index of $\phi_\timestep$ for some $h>0$ is the existence
of a common index pair for both dynamical systems.
The following lemma was verified in the proof of~\cite[Theorem 2]{MM1}.
\begin{lemma}
\label{lemma_index_pairs_flow_map}
Let $S$ be an isolated invariant set for $\phi_\timestep$ (hence for $\phi$).
Then there is a pair $Q'=(Q_1,Q_2)$ of compact sets such that $Q$
is an index pair for $(S,\phi)$ and an index pair for $(S,\phi_\timestep)$.
\end{lemma}

We also have a slightly more general statement for
a time step function $\stepfunction$, but we need an extra assumption
about the equality of the invariant sets with respect to $\phi$ and $\phi_\stepfunction$.
\begin{lemma}
\label{lemma_index_pairs_flow_map_n}
Let $S$ be an isolated invariant set for $\phi_\stepfunction$ and also for $\phi$. Then
there are compact sets $Q_1,Q_2$ and $M$ such that
\enumeration{
\item $Q=(Q_1,Q_2)$ is an index pair for $(S,\phi)$ and for $(S,\phi_\stepfunction)$ in $M$, and
\item the index map $I_Q$ is the identity on $\Ho{Q_1,Q_2}$.
}
\end{lemma}
\begin{proof}
The proof is similar to the proof of Lemma~\ref{lemma_index_pairs_flow_map} as given in
\cite[Theorem 2]{MM1}: The proof of \cite[Chapter I, Theorem 5.1]{Ryba} shows the existence
of an open neighborhood $V\subset Y$ of $S$ such that $\cl V$ isolates $S$ and of continuous functions
$\kappa,\lambda \colon V \to [0,\infty)$ such that $S = \kappa^{-1}(0) \cap \lambda^{-1}(0),$ and
\begin{enumerate}[label=(\alph*)]
\item If $\kappa(x)>0, t>0$ and $\phi(t,x) \in V$, then $\kappa(x)<\kappa(\phi(t,x));$
\item if $\lambda(x)>0, t>0$ and $\phi(t,x)\in V$, then $\lambda(x)>\lambda(\phi(t,x)).$
\end{enumerate}
For an arbitrary $\zeta>0$, define subsets of $V$:
\begin{align*}
G(\zeta) := \{ x \in V \mid \kappa(x) < \zeta, \lambda(x) < \zeta \}, \\
H(\zeta) := \{ x \in V \mid \kappa(x) \leq \zeta, \lambda(x) \leq \zeta \}.
\end{align*}
Let $T := \max\setof{ \stepfunction(x) \mid x \in \cl V }$.
The local compactness of $Y$ can be used to observe that for any open neighborhood $U$ of $S$
there is a $\zeta>0$ such that $\phi([0,T],H(\zeta)) \subset U$. Applying this observation to $U=\Int V$,
we conclude the existence of an $\epsilon >0$ such that $\phi([0,T],H(\epsilon)) \subset V$.
Let $M := H(\epsilon)$. Applying the observation to $U=G(\epsilon)$ shows the existence of a
$\delta>0$ such that $\phi([0,T],H(\delta)) \subset G(\epsilon).$ Define
\begin{align*}
Q_1 &:= \{ x \in V \mid \kappa(x)\leq\epsilon, \lambda(x) \leq \delta \}, \\
Q_2 &:= \{ x \in V \mid \delta \leq \kappa(x) \leq \epsilon, \lambda(x) \leq \delta \} \subset Q_1.
\end{align*}
The pair $(Q_1,Q_2)$ is an index pair for $(S,\phi)$ and for $(S,\phi_\stepfunction)$
in $M$.
It is straightforward to check the properties in Definitions \ref{defIndexPairFlow} and
\ref{defIndexPairMap}, hence (i) is proven.

To see (ii), we note that
\[ (Q_1,Q_2) \times [0,1] \ni (x,s) \mapsto \phi(s\stepfunction(x),x) \in
(Q_1 \cup (Y \setminus \Int{M}), Q_2 \cup (Y \setminus \Int{M})) \]
is a homotopy from $i_Q$ to $f_Q$. Hence the index map $I_Q$ is the identity.
\end{proof}

\begin{theorem}
\label{thm:indexFlow}
Let $S$ be an isolated invariant set for $\phi$ and $\phi_\stepfunction$.
Let $(P_1,P_2)$ be an index pair
for $(S,\phi_\stepfunction)$ and $I_P$ its index map. Then there is an
isomorphism
\[ \Con{S,\phi} \cong \dom\Leray(I_P). \]
\end{theorem}

\begin{proof}
Consider a common index pair $(Q_1,Q_2)$ for $(S,\phi)$ and
$(S,\phi_\stepfunction)$ in $M$ as in Lemma~\ref{lemma_index_pairs_flow_map_n}.
Then $\Leray(I_P)$ and $\Leray(I_Q)$ are conjugate and therefore
\[\dom\Leray(I_P) \cong \dom\Leray(I_Q) = \dom I_Q = \Ho{Q_1,Q_2}.\]
\end{proof}

\subsubsection{Shift equivalences}
\label{ssec:shift-equi}
Another equivalence relation commonly used on the index map $I_P$ is the following one proposed in~\cite{RS}.

\begin{definition}
\label{def:shift-equi}
Let $\alpha\colon V \to V$ and $\beta\colon W \to W$ be endomorphisms of
graded modules. They are \emph{shift equivalent} if there are homomorphisms
$r \colon V \to W$ and $s\colon W \to V$ such that
$\beta\circ r = r\circ \alpha$, $s \circ \beta = \alpha\circ s$,
$s\circ r = \alpha^n$ and $r \circ s = \beta^n$ for some $n \in \NN$.
\end{definition}

\begin{theorem}{\rm(\cite[Theorem 3.29]{MiMo}, \cite[Lemma 4.3]{Sz})}
\label{thmShiftEqui}
If $P$ and $P'$ are two index pairs around the same isolated invariant set
for a discrete dynamical system, then $I_P$ and $I_{P'}$ are shift equivalent.
\end{theorem}

The following theorem is similar to, but slightly stronger than Theorem~\ref{thm:indexFlow}.
\begin{theorem}
\label{thm:indexFlowShift}
Let $S$ be an isolated invariant set for $\phi$ and $\phi_\stepfunction$.
Let $(P_1,P_2)$ be an index pair
for $(S,\phi_\stepfunction)$ and $I_P$ its index map. Then there is an
isomorphism
\[ \Con{S,\phi} \cong \Ho{P_1,P_2}/\gker(I_P). \]
\end{theorem}
\begin{proof}
There is a common index pair $(Q_1,Q_2)$ for $(S,\phi)$ and
$(S,\phi_\stepfunction)$ due to Lemma~\ref{lemma_index_pairs_flow_map_n}.
Thus, by Theorem~\ref{thmShiftEqui}, there are homomorphisms $r,s$ and an $n\in\NN$ such that the
following diagram commutes.
\[ \begin{tikzcd}
\Ho{P_1,P_2} \ar{d}[left]{r} \ar{r}{I_P^n} & \Ho{P_1,P_2} \ar{d}{r} \\
\Ho{Q_1,Q_2} \ar{r}[below]{I_Q^n=\id} \ar{ur}{s} & \Ho{Q_1,Q_2}
\end{tikzcd} \]
The lower right triangle shows that $s$ is injective and $r$ is surjective.
This yields
\[ \Con{S,\phi}\cong \Ho{Q_1,Q_2} \cong \im(s) = \im(I_P^n)
   \cong \Ho{P_1,P_2}/\ker(I_P^n).\]
Additionally, $\ker(I_P^n) = \gker(I_P)$ because $I_P^{kn}=I_P^n$ for any $k>0$.
\end{proof}

\subsection{Morse decompositions}
\label{sec:MD}

\begin{definition}
Let $y \in Y$.
\enumeration{
  \item The $\omega$-\emph{limit set} $\omega(y,\phi)$ is
    the set of accumulation points of $\phi([0,\infty),y)$.
  \item The $\alpha$-\emph{limit set} $\alpha(y,\phi)$ is
    the set of accumulation points of $\phi((-\infty,0],y)$.
}
\end{definition}

\begin{definition}
Let $f\colon Y \to Y$ be a discrete dynamical system.
Let $y \in Y$ be such that there is a solution $\ZZ \to Y$ through $y$.
\enumeration{
  \item The $\omega$-\emph{limit set} $\omega(y,f)$ is the set of accumulation points of $\{f^k(y) \mid k>0\}$.
  \item A point $x \in Y$ is in the $\alpha$-\emph{limit set} $\alpha(y,f)$ iff
  there exists a solution $\gamma \colon \ZZ \to Y$ with $\gamma(0)=y$ such that $x$
  is an accumulation point of $\gamma((-\infty,0])$.
}
\end{definition}

\begin{definition}
\label{def:MorseDecomp}
Given a dynamical system (i.e., a flow $\phi$ or a map $f$) and an isolating neighborhood
$X$ with $S=\Inv(X)$ (denoting $\Inv(X,\phi)$ or $\Inv(X,f)$, respectively),
we call a set of disjoint isolated invariant sets $\{ M_p \mid p \in \cP \}$
together with an acyclic directed graph $\mathrm{MG}$ with vertices $\cP$ a
\emph{Morse decomposition in $X$} if for every $y \in S$ one of the following holds:
\enumeration{
  \item $y \in M_p$ for a certain $p \in \cP$; or
  \item there are $p,q \in \cP$ and a path from $p$ to $q$ in $\mathrm{MG}$,
  such that $\alpha(y) \subset M_p$ and $\omega(y) \subset M_q$.
}
\end{definition}

In order to show that a certain Morse decomposition for $\phi_\stepfunction$ is
also a Morse decomposition for $\phi$, we can apply the following two lemmata.

For $A \subset Y$, let
\[ \phi_{[0,\stepfunction]}(A) := \bigcup_{x \in A} \phi([0,\stepfunction(x)],x) \subset Y. \]

\begin{lemma}
\label{criterionEqualInv}
Let $\{ M_p \mid p \in \cP \}$ be a Morse decomposition in $X$
for $\phi_\stepfunction$ with isolating neighborhoods
$N_p$, i.e., $N_p \cap N_q = \varnothing$ for $p \neq q$ and
$M_p = \Inv(N_p,\phi_\stepfunction) \subset \Int{N_p}$ for all $p\in\cP$. Let $p \in \cP$. If
\[ (*) \quad \quad \phi_{[0,\stepfunction]}(N_p) \subset X \text{ and }\phi_{[0,\stepfunction]}(N_p) \cap N_q = \varnothing \text{ for all } q \in \cP \setminus \{p\}, \]
then $\Inv(N_p,\phi) = \Inv(N_p,\phi_\stepfunction).$
\end{lemma}

\begin{proof}
Let $p \in \cP$ and
$ N_p' = \cl \left(X \setminus \bigcup_{q\in \cP\setminus\{p\}} N_q\right). $
First we show
\begin{equation}
\label{criterionEqualInvEq1}
  \Inv(N_p, \phi_\stepfunction) \subset \Inv(N_p',\phi).
\end{equation}
Let $\gamma\colon \ZZ \to N_p$ be a solution to $\phi_\stepfunction$ in $N_p$. Then for each $n\in\ZZ$:
$\phi([0,\stepfunction(\gamma(n))],\gamma(n)) \subset N_p'$ by assumption $(*)$ in the lemma. Gluing these pieces together yields a
trajectory for $\phi$ in $N_p'.$

We also show
\begin{equation}
\label{criterionEqualInvEq2}
   \Inv(N_p',\phi_\stepfunction) \subset \Inv(N_p,\phi_\stepfunction).
\end{equation}
Assume there is a point
$x \in \Inv(N_p',\phi_\stepfunction)\setminus \Inv(N_p,\phi_\stepfunction)$.
Hence, $x \notin \cup_{q\in P} M_q$. Now either $\alpha(x,\phi_\stepfunction)$ or $\omega(x,\phi_\stepfunction)$
must lie in some $M_q$ with $q \neq p$
because they cannot lie within the same Morse set.
But this means that any solution of $\phi_\stepfunction$ through $x$
has to contain points in $\Int N_q$, which is disjoint from $N_p'.$
We conclude $x \notin \Inv(N_p',\phi_\stepfunction)$, a contradiction.

Overall, we get the following inclusions, where the middle one is trivial.
\[ M_p \stackrel{(\ref{criterionEqualInvEq1})}{\subset} \Inv(N_p', \phi)
   \subset \Inv(N_p',\phi_\tau) \stackrel{(\ref{criterionEqualInvEq2})}{\subset} M_p. \]
Each set is the same subset of $N_p$. Therefore also
$\Inv(N_p',\phi) = \Inv(N_p,\phi).$
\end{proof}

\begin{lemma}
\label{lemmaMorseFlow}
Let $X\subset Y$ be an isolating neighborhood for $\phi_\tau$ (and
hence for $\phi$).
Let $\{ M_p \mid p \in \cP \}$ be a Morse decomposition for $\phi_\stepfunction$
in $X$ with
Morse graph $\mathrm{MG}$ and assume that each $M_p$ is invariant also for $\phi$.
Then $\{ M_p \mid p \in \cP \}$ is also a Morse decomposition
for $\phi$ in $X$ with Morse graph $\mathrm{MG}$.
\end{lemma}

\begin{proof}
We only need to show that the Morse graph $\mathrm{MG}$ is preserved.
From here, consider a point $y \in \Inv(X,\phi) \subset \Inv(X,\phi_\tau)$.
There are $p,q \in \cP$
and a directed path from $p$ to $q$ in $\mathrm{MG}$ such that
$\alpha(y,\phi_\stepfunction) \subset M_p$
and $\omega(y,\phi_\stepfunction) \subset M_q.$
We show that $\omega(y,\phi) \subset M_q$ by contradiction.
The analogous statement for the $\alpha$-limit is proven similarly.

Let $N_q$ be an isolating neighborhood for $\phi_\stepfunction$ and therefore
for $\phi$ around $M_q$,
hence $\Inv(N_q,\phi)=\Inv(N_q,\phi_\stepfunction)=M_q \subset \Int N_q.$
We continue analogously to the proof of Theorem~\ref{theorem1n}. We define a function
 \[
 \sigma\colon N_q\ni x\mapsto \sup\setof{t\in \RR_{\geq 0}\mid \phi([0,t],x)\subset N_q}\in [0,\infty].
 \]
Let $T = \max \stepfunction(N_q)$.
Now there is a compact neighborhood $\widetilde{N_q}$ of
$M_q$ such that $\sigma(x) \geq T$
for all $x \in \widetilde{N_q}$. For all $n\in \NN$,
let $\gamma(n) := \phi_\stepfunction^n(y).$

Assume that $\omega(y,\phi)$ contains a point $y'$ outside of $N_q$.
We construct a subsequence of $\gamma$ as follows:
Since $\widetilde{N_q}$ is a neighborhood of $\omega(y,\phi_\stepfunction)$
and $\gamma(\NN)\cap M_q = \varnothing$, there is an
$\tilde{n}\geq 0$ such that
$\gamma(\tilde{n}) \in \widetilde{N_q} \setminus M_q$.
Let $s>0$ be such that $\gamma(\tilde{n}) = \phi(s,y)$.
There is an $s'\geq s$ such that $\sigma(\phi(s',y))=T$, and then
$\phi([s',s'+T],x) \subset \cl(N_q \setminus \widetilde{N_q})$.
Hence, there is an $n_0 \geq \tilde{n}$ such that
$\gamma(n_0) \in \cl(N_q \setminus \widetilde{N_q})$.

Going on like this, we construct a subsequence $\gamma(n_0), \gamma(n_1), \ldots$ of
points in $\cl(N_q \setminus \widetilde{N_q})$.
It has a converging subsequence whose limit
is also in $\omega(y,\phi_\stepfunction)\subset M_q$. A contradiction.

Overall, $\omega(y,\phi) \subset N_q$ and therefore
$\omega(y,\phi) \subset \Inv(N_q,\phi) = M_q$ because $\omega$-limits are invariant.
\end{proof}

We are ready to show
\begin{theorem}
\label{theoremMorseDecompCriteria}
Let $X$ be an isolating neighborhood for $\phi_\stepfunction$ for an arbitrary
continuous function $\stepfunction \colon Y \to \RR^+$ and
$\{ M_p \mid p \in \cP \}$ a Morse decomposition for $\phi_\stepfunction$ in $X$
with isolating neighborhoods $\{N_p\}$.

Suppose that either
\begin{enumerate}[label=(\Alph*)]
  \item \label{critFixedStep} the function $\stepfunction$ is constant, i.e., $\stepfunction(x)=h$ for all $x\in Y$; or
  \item \label{critVaryingStep} for all $p \in \cP: \phi_{[0,\stepfunction]}(N_p) \subset X$
  and for all $q \neq p$ holds $\phi_{[0,\stepfunction]}(N_p) \cap N_q = \varnothing$.
\end{enumerate}
Then $\{M_p\}$ is also a Morse decomposition for $\phi$ in $X$ with the same Morse graph.
\end{theorem}

\begin{proof}
The case~\ref{critFixedStep} of a constant time step follows from
Corollary~\ref{theorem_isolation} and
Lemma~\ref{lemmaMorseFlow}. The case~\ref{critVaryingStep} follows from
Lemmata~\ref{criterionEqualInv} and~\ref{lemmaMorseFlow}.
\end{proof}

\begin{remark}
\label{numericsCriterionB}
Criterion~\ref{critVaryingStep} enables us to verify numerically
that the output from the algorithm in Section~\ref{ssec:variableStep}
does indeed describe a Morse decomposition:
For each $p \in \cP$, we construct an enclosure
$\cZ(p)$ of all the trajectories $\phi([0,\tau(x)],x),\ x \in |\cN(p)|,$
such that $\cZ(p) \subset \cX \setminus \cup_{q \neq p} \cN(q)$.
Our algorithm uses CAPD routines to construct a good enclosure
$\cZ(p)$ with
\[ \bigcup_{\xi\in\cN(p)} \phi([0,\max \stepfunction(\xi)],\xi) \subset |\cZ(p)|, \]
In Figure~\ref{SubfigCriterionInv}, each set $\cN(p)$
is shown with a darker collar around it such that
together they form $\cZ(p)$.
In our example, the algorithm successfully checked
$\cZ(p) \subset \cX$ and
$\cZ(p) \cap \cN(q) = \varnothing$ for any $p \neq q$.
The additional checks for criterion~\ref{critVaryingStep} took only about a second.
\end{remark}

The following example shows that when~\ref{critFixedStep} is not fulfilled,
we need some kind of check that a Morse set $M_p$ for
$\phi_\stepfunction$ is also invariant under $\phi$.

\begin{remark}
\label{remark_counterexample}

Let $S^1 = \{ (\cos \theta, \sin \theta) \mid \theta \in [0,2\pi) \} \subset \RR^2$
be the phase space $X=Y$ and let the flow $\phi$ be induced by $\dot{\theta}=1$.

For every point $\theta \in S^1$, its limit sets are
$\omega(\theta,\phi) = S^1$ and $\alpha(\theta,\phi) = S^1$.
The sets $\varnothing$ and $S^1$
are the only subsets invariant for $\phi$. Define the time step function
\[ \stepfunction \colon S^1 \ni \theta \mapsto \sin\theta + 2\pi \in \RR^+, \]
which is well-defined and continuous since it has period $2\pi$.
Figure~\ref{fig:counterexample} shows a plot of $\stepfunction$
and a trajectory of $\phi_\stepfunction$.

By definition, $\stepfunction(0) = \stepfunction(\pi) = 2 \pi$,
$2\pi < \stepfunction((0,\pi))< 3\pi$
and $\pi < \stepfunction((\pi,2\pi))<2\pi$.
Additionally $|\stepfunction'(\theta)| < 1$ for $\theta\neq 0,\pi$.
These properties suffice to see that for any $0<\epsilon<\pi$ the subsets
$N_1 = [-\epsilon,+\epsilon]$ and $N_2 = [\pi-\epsilon,\pi+\epsilon]$
are isolating neighborhoods for $\phi_\stepfunction$ with isolated invariant sets $M_1 = \{0\}$
and $M_2=\{\pi\}$.
Whenever $\theta\notin\{0,\pi\}$, then $\alpha(\theta,\phi_\stepfunction)=M_1$
and $\omega(\theta,\phi_\stepfunction)=M_2$.
We do therefore get a Morse graph $1 \to 2$ for $\phi_\stepfunction$, an attractor-repeller pair.
But there is no
attractor-repeller pair for $\phi$ in which both invariant sets are non-empty.\footnote{R. Vandervorst provided a similar example, independently, to the first author.}

\begin{figure}
  \begin{subfigure}[b]{0.4\textwidth}
  \centering
  \resizebox{\textwidth}{!}{
  \begin{tikzpicture}[scale=1]
    \begin{axis}
      [
      xlabel={\LARGE $\theta$},
      xmin=-0.2,
      xmax=6.48318,
      xtick={0, 1.5708, 3.14159, 4.7123889, 6.28318 },
      xticklabels={ \LARGE 0, \LARGE $\frac{\pi}{2}$, \LARGE $\pi$, \LARGE $\frac{3\pi}{2}$, \LARGE$2\pi$ },
      ytick={0, 3.14159, 5.28318, 6.28318, 7.28318 },
      yticklabels={ \LARGE $0$, \LARGE $\pi$, \LARGE $2\pi-1$, \LARGE $2\pi$, \LARGE $2\pi+1$ },
      ymin=-0.5,
      extra y ticks = {5.28318, 6.28318, 7.28318 },
      extra y tick labels = ,
      extra y tick style  = { grid = major },
    ]
    \addplot [
      domain=-0.5:2*pi+0.5,
      samples=100,
      very thick
    ]
    {sin(deg(x))+2*pi)};
    \node at (360,730) {\LARGE$\stepfunction$};
    \end{axis}
  \end{tikzpicture}
  }
  \caption{Time step function $\stepfunction$ on $S^1$}
  \end{subfigure}%
  \quad
  \begin{subfigure}[b]{0.45\textwidth}
  \centering
  \resizebox{\textwidth}{!}{
  \begin{tikzpicture}[scale=1.8]
    \draw[very thick,|-|] (0.8660254,-0.5) arc (-30:30:1.0);
    \draw (1.15,0.4) node{\large$N_1$};
    \draw[very thick,|-|] (-0.8660254,-0.5) arc (210:150:1.0);
    \draw (-1.15,0.4) node{\large$N_2$};
    \draw[fill=black] (-1,0) circle [radius=0.05];
    \draw (1.13,0) node{\large$0$};
    \draw[fill=black] (1,0) circle [radius=0.05];
    \draw (-1.13,0) node{\large$\pi$};
    \draw[
      decoration={
      markings,
      mark=at position 0.2 with {\arrow{triangle 45}},
      mark=at position 0.53 with {
        \draw[fill=white] (0,0) circle [radius=0.05];
        \draw node[left]{\footnotesize$\gamma(3)$};},
      mark=at position 0.66 with {
        \draw[fill=white] (0,0) circle [radius=0.05];
        \draw node[below left]{\footnotesize$\gamma(2)$};},
      mark=at position 0.74 with {\arrow{triangle 45}},
      mark=at position 0.81 with {
        \draw[fill=white] (0,0) circle [radius=0.05];
        \draw node[below]{\footnotesize$\gamma(1)$};},
      mark=at position 0.90 with {
        \draw[fill=white] (0,0) circle [radius=0.05];
        \draw node[below right]{\footnotesize$\gamma(0)$};},
      mark=at position 0.95 with {
        \draw[fill=white] (0,0) circle [radius=0.05];
        \draw node[right]{\footnotesize$\gamma(-1)$};},
      },
      postaction={decorate}
    ]
    (0,0) circle [radius=1];  
  \end{tikzpicture}
  }
  \caption{Attractor-repeller pair for $\phi_\stepfunction$}
  \end{subfigure}%
\caption{The example from Remark~\ref{remark_counterexample}. The time step function on the left
admits the attractor-repeller pair for $\phi_\stepfunction$ in the right figure
with $N_1$ isolating the
repelling fixed point $0$ and $N_2$ isolating the attracting fixed point $\pi$.
The arrows represent the direction
of the flow $\phi$. A part of a typical trajectory $\gamma \colon \ZZ \to S^1$
of $\phi_\stepfunction$
in the lower half of the circle is drawn.}
\label{fig:counterexample}
\end{figure}
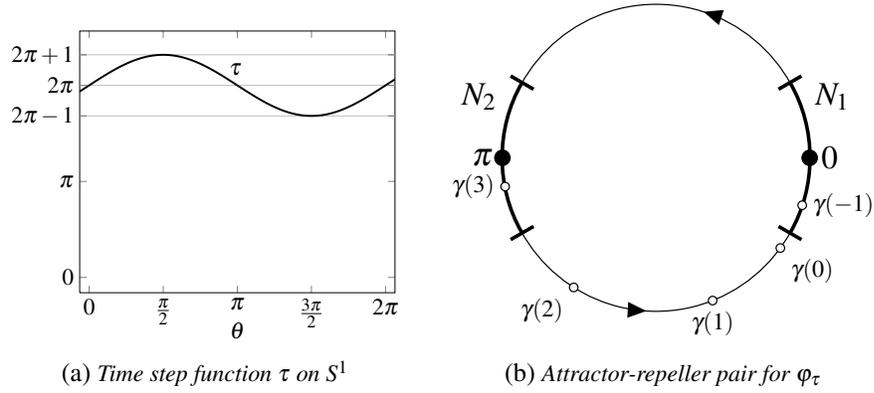
\end{remark}

The example shows that the assumption that $M_p$ be invariant also for the flow is necessary
in Lemma~\ref{lemmaMorseFlow}. Such a function $\stepfunction$ can always be constructed when
the flow has a limit cycle $L$. One just has to replace $2\pi$ by the period
of the limit cycle and extend $\stepfunction$ from $L$ to the whole phase space $Y$
using Tietze's extension theorem.

\section{Final remarks}

Using a time step function varying in space gives a lot of flexibility when applying
rigorous algorithms for finding Morse decompositions.
Sometimes, the step function $\stepfunction$ is the only way to
obtain a meaningful Morse decomposition numerically. Additionally, the heuristic
presented can be justified more naturally than a specific choice of fixed time step. Up to now,
we cannot see from a given equation which strategy is more promising. It would be interesting
to find a more systematic approach to find out when to use which strategy.

\end{document}